%% file: main.tex
\documentclass{article}
\usepackage{amsfonts}
\usepackage{amsmath}
\usepackage{amssymb}
\usepackage{amsthm}
\usepackage{url}

\usepackage{tikz-cd}

\renewcommand{\int}{\operatorname{int}}
\newcommand{\cl}{\operatorname{cl}}
\newcommand{\term}{\textit}

\theoremstyle{theorem}
\newtheorem{theorem}{Theorem}
\newtheorem{proposition}[theorem]{Proposition}
\newtheorem{lemma}[theorem]{Lemma}
\newtheorem{corollary}[theorem]{Corollary}

\theoremstyle{definition}
\newtheorem{definition}[theorem]{Definition}

\newtheorem{example}[theorem]{Example}
\newtheorem{question}[theorem]{Question}

\begin{document}

\title{Non-Hausdorff $T_1$ Properties}
\markright{Non-Hausdorff $T_1$ Properties}
\author{Steven Clontz}  %DO NOT FILL IN AUTHOR'S NAMES UNTIL YOU RECEIVE YOUR PROVISIONAL ACCEPT LETTER. SUBMISSIONS TO THE MONTHLY ARE DOUBLE BLIND.

\maketitle

\input{content}

\bibliographystyle{vancouver}
\bibliography{sample}

% After you receive your provisional accept letter, you need to re-submit your manuscript with the following information  (commented out in the blind submission) filled in for each author. Just remove the \% comment coding, fill out the information, and re-submit your manuscript.

%\begin{biog}
%\item[Author Name 1] Insert author bio here.
%\begin{affil}
%Department of Mathematics, University America, Washington DC 20036\\
%authorname@ua.edu
%\end{affil}
%\end{biog}

%\begin{biog}
%\item[Author Name 2] Insert author bio here.
%\begin{affil}
%Department of Mathematics, University America, Washington DC 20036\\
%authorname@ua.edu
%\end{affil}
%\end{biog}

\vfill\eject

\end{document}

%% file: content.tex
\begin{abstract}
Several weakenings of the \(T_2\) property for
topological spaces, including \(k\)-Hausdorff, \(KC\),
weakly Hausdorff, semi-Hausdorff, \(RC\), and \(US\),
have been studied by mathematicians. Here we provide a complete
survey of how these properties do or do not relate to one another,
including several new results to fill in the gaps in the existing
literature, motivated by the use of a community-maintained
database of topological spaces and their properties.
\end{abstract}

The following properties were introduced in
Wilansky's
``Between \(T_1\) and \(T_2\)''
\cite{MR0208557}, appearing in the \textit{Monthly}
in 1967.

\begin{definition}[P99\footnote{Property ID as assigned in the \(\pi\)-Base \cite{pibase}}]
    A space is \term{US} (``Unique
    Sequential limits'')
    provided that every
    convergent sequence has a unique limit.
\end{definition}

\begin{definition}[P100]
    A space is \term{KC} 
    (``Kompacts are Closed'')
    provided that
    its compact subsets are closed.
\end{definition}

The following fact is observed there.

\begin{proposition}
\[T_2\Rightarrow KC\Rightarrow US\Rightarrow T_1\]
with no implications reversing.
\end{proposition}

In the process of modeling this result from 1967
in the \textit{\(\pi\)-Base community database of
topological counterexamples} \cite{pibase} (an
open-source web application serving as a
modern adaptation of \cite{steen1978counterexamples})
several other examples of properties in the
literature properly implied by \(T_2\) and
properly implying \(T_1\) were encountered. This paper
aims to provide a modern update to Wilansky's 1967 manuscript
by surveying these non-Hausdorff $T_1$ properties.

% These results are a combination of retooled work from the existing
% peer-reviewed literature, proofs aired out in online forums such
% as \textit{MathOverflow}, and original results of the author.
% In particular, we aim to illustrate the \(\pi\)-Base as a tool
% for advancing mathematics, and Section \ref{sec-infra}
% provides some details on how modern sociotechnical
% infrastructure for mathematics research is transforming, and will
% transform, the human experience of mathematics research in general.

\section{\(k\)-Hausdorff}

The following property was introduced in
\cite{MR0374322}; here we use an equivalent characterization
(but will revisit the original definition in Section \ref{k2Hsec}).

\begin{definition}[P170]\label{k1H}
A space is \term{\(k_1\)-Hausdorff} or \(k_1H\)
provided each compact subspace is \(T_2\).\footnote{
The subscript in \(k_1\) will be explained in section \ref{k2Hsec}.}
\end{definition}

Since \(T_2\) is a hereditary property,
it's immediate that all \(T_2\) spaces
are \(k_1H\). In fact, we have the following.

\begin{theorem}
\[T_2\Rightarrow k_1H\Rightarrow
KC.\]
\end{theorem}
\begin{proof}
This follows from Theorem 2.1 of \cite{MR0374322}.
To see the second implication, let \(X\) be
\(k_1H\) and let \(K\subseteq X\)
be compact. Then for each \(x\in K\),
the set \(Y=K\cup\{x\}\) is compact, and thus \(T_2\).
Then \(K\) is a compact subset of \(T_2\)
\(Y\), and thus closed in \(Y\). It follows that
\(\{x\}\) is open in \(Y\) and not a limit point
of \(K\) in \(X\).
\end{proof}

Example 3.6 of \cite{MR0374322} demonstrates
that the first implication does not reverse,
even when topologies are assumed to be
generated by compact subspaces (P140), by taking a quotient of
the tangent disc topology
(S74\footnote{Space ID as assigned in the \(\pi\)-Base.}).

Without requiring spaces be P140,
an elementary counterexample is available.

\begin{definition}[P136]
A space is \term{anticompact} provided all
compact subsets are finite.
\end{definition}

\begin{lemma}
All anticompact \(T_1\) spaces are
\(k_1H\).
\end{lemma}
\begin{proof}
Each compact subset is a finite \(T_1\) space.
All finite \(T_1\) spaces are discrete and thus
\(T_2\).
\end{proof}

\begin{definition}[P39]
A space is \term{hyperconnected} provided
every nonempty open set is dense.
\end{definition}

\begin{example}[S17]\label{S17}
The co-countable topology on an uncountable
set is \(k_1H\) but not \(T_2\). To see this,
first observe that the space is hyperconnected:
every pair of non-empty open sets intersects.
Therefore the space is not \(T_2\),
but the space is still \(k_1H\), which
follows from the fact that the space is
anticompact and \(T_1\), and the preceding lemma.
\end{example}

Any example of a compact non-\(T_2\)
space provides
an example of a non-$k_1H$ space, such as
the following.

\begin{definition}
The \term{one-point compactification} of a space
\(X\) is \(X\cup\{\infty\}\) where points of \(X\)
have their usual neighborhoods, and neighborhoods
of \(\infty\) are complements of closed and compact
subsets of \(X\).
\end{definition}

\begin{example}[S29]\label{S29}
The one-point compactification of the rationals
is \(KC\) but not \(k_1H\),
because the one-point compactification
of any \(KC\), compactly-generated,
non-locally compact space is
\(KC\) but not \(T_2\): see Corollary 1 of
\cite{MR0208557}.
\end{example}

\begin{example}[S97]\label{S97}
Counterexample \#99 from 
\cite{steen1978counterexamples} is
the set \(\omega^2\cup\{\infty_x,\infty_y\}\) with \(\omega^2\)
discrete. Neighborhoods of \(\infty_x\) contain
all but finitely-many elements from each row of
\(\omega^2\), and neighborhoods of \(\infty_y\) contain
all but finitely-many entire rows of \(\omega^2\).

Compact and not-\(T_2\) are immediate. To see
that compacts are closed, note that a non-closed
set must have either \(\infty_x\) or \(\infty_y\) as a limit
missing from the set. If \(\infty_x\) is the
missing limit, the set must contain infinitely-many
points from some row, and then
the open cover of the set by a neighborhood of
\(\infty_y\) missing that row and singletons of \(\omega^2\)
otherwise has no finite subcover. If \(\infty_y\) is the 
missing limit, the set must contain at least one
point from infinitely-many rows, and then the
open cover of the set by a neighborhood of
\(\infty_x\) missing those points and the singletons
of \(\omega^2\) otherwise has no finite subcover.
\end{example}

\section{Weakly Hausdorff}

The following property was introduced in
\cite{mccord1969classifying}.

\begin{definition}[P143]
A space is \term{weakly Hausdorff} or
\(wH\) provided
the continuous image of any compact Hausdorff
space into the space is closed.
\end{definition}

In that paper it was observed that \(T_2\) was
sufficient and \(T_1\) was necessary for this
property. In fact, this can be tightened.

\begin{theorem}
\[KC\Rightarrow wH\Rightarrow US.\]
\end{theorem}
\begin{proof}
For the first implication, note that
the continuous image of compact is compact,
so if compacts are closed, the continuous
image of compact \(T_2\) spaces are closed.

The second implication was originally shown in
\cite{mathse4267169}. We will demonstrate a
slightly different proof by showing the
contrapositive. If a space \(Y\) is not \(US\),
then there must be some sequence \(a_n\) converging
to two distinct points \(x\) and \(y\).

Suppose
\(y\not=a_n\) for any \(n<\omega\). Consider
the compact Hausdorff space given by
\(X=\{0\}\cup\{2^{-n}:n<\omega\}\subseteq\mathbb R\).
Consider \(f:X\to Y\) where \(f(0)=x\) and
\(f(2^{-n})=a_n\). To prove continuity, we need
only worry about preimages of the neighborhoods
of \(x\), but by
definition any open neighborhood of \(x\) contains
cofinite \(a_n\), so its preimage in \(X\) is open.
Then \(f[X]=\{x\}\cup\{a_n:n<\omega\}\)
is not closed as it's missing its
limit point \(y\).

If \(y=a_N\), first consider the subsequence of
\(a_n\) missing \(y\). If this subsequence is
infinite, then the above proof may be repeated
using the subsequence instead. If the subsequence
is finite, then \(a_n=y\) for all but finitely-many
\(n\). Then \(\{y\}\) is a compact Hausdorff space
embedded in \(Y\) that is not closed, as
\(x\) is a limit point of that set.
\end{proof}

Brian Scott showed in \cite{964546} that
\(wH\) is productive, and in
\cite{466564} that the square of \(KC\) need not
be \(KC\). So the square of such a space
is \(wH\) but not \(KC\). We provide
the details here for the convenience of the reader.

\begin{lemma}\label{wHimage}
The continuous image of a compact \(T_2\) space
in a \(wH\) space is \(T_2\).
\end{lemma}
\begin{proof}
Consider a compact \(T_2\) space
\(X\) and a continuous map
\(f:X\to Y\) with \(Y\) \(wH\).

Then pick \(y_0,y_1\in f[X]\).
Note that as \(Y\) is \(T_1\),
\(f^\leftarrow[\{y_0\}]\)
and \(f^\leftarrow[\{y_1\}]\) are disjoint closed
subsets of a compact \(T_2\), and therefore \(T_4\),
space. Separate these by open \(U_0,U_1\subseteq X\).

Then \(X\setminus U_1,X\setminus U_0\) are compact
and \(T_2\),
and thus \(f[X\setminus U_1],f[X\setminus U_0]\)
are closed. Consider
\(V_0=f[X]\setminus f[X\setminus U_1]\) and
\(V_1=f[X]\setminus f[X\setminus U_0]\). These sets
are open in \(f[X]\), with \(y_0\in V_0\) and
\(y_1\in V_1\). Note also that if \(z\in V_0\cap V_1\),
then there is some \(x\in X\) with \(f(x)=z\);
\(f(x)\not\in f[X\setminus U_0]\), and thus
\(x\not\in X\setminus U_0\);
\(f(x)\not\in f[X\setminus U_1]\), and thus
\(x\not\in X\setminus U_1\). But then
\(x\in U_0\cap U_1\), a contradiction.
\end{proof}

\begin{lemma}
The product of \(wH\) spaces is \(wH\).
\end{lemma}
\begin{proof}
Let \(\mathcal Y\) be a collection of \(wH\) spaces,
\(X\) be compact \(T_2\), and
\(f:X\to\prod\mathcal Y\) be continuous. For each
\(Y\in\mathcal Y\), the projection \(\pi_Y\) is
continuous, so \(\pi_Y\circ f:X\to Y\) is continuous.
Since \(Y\) is \(wH\), \(\pi_Y\circ f[X]\) is
closed, compact, and \(T_2\).

It follows that
\(Z=\prod\{\pi_Y\circ f[X]:Y\in\mathcal Y\}\)
is a closed, compact, and \(T_2\) subspace of
\(\prod\mathcal Y\). Importantly,
\(f[X]\subseteq Z\) is a compact subspace
of a \(T_2\) space, and thus closed in \(Z\).
Therefore, \(f[X]\) is closed in
\(\prod\mathcal Y\).
\end{proof}

\begin{theorem}\label{non-kc-square}
The square of a compact non-\(T_2\) space
is not \(KC\).
\end{theorem}
\begin{proof}
The diagonal of a space is homeomorphic to the space,
so the diagonal of a compact space is
compact. But a space is \(T_2\) if and only if
its diagonal is closed, so the diagonal of a
compact non-\(T_2\) space is compact but not
closed.
\end{proof}

\begin{example}[S31]\label{S31}
The square of the
one-point compactification of the rationals
is \(wH\), because the one-point compactification
of the rationals is \(KC\) and therefore
\(wH\).

But since the one-point compactification of
the rationals is not \(T_2\),
its square is not \(KC\) by the previous
theorem.
\end{example}

A slight modification of a common counterexample
yields a space which is US but not \(wH\).

\begin{example}[S23]\label{S23}
The \term{Arens-Fort space} is
\(\omega^2\cup\{z\}\) where \(\omega^2\) is
discrete, and neighborhoods of \(z\) contain
all but a finite number of points from all but
a finite number of columns. This space is \(T_2\)
and thus \(KC\). It is also not
locally compact: it is anticompact while \(\infty\)
has no finite neighborhoods.
\end{example}

\begin{example}[S165]
The one-point compactification \(X=Y\cup\{\infty\}\)
of the Arens-Fort
space \(X\) is \(US\) but not \(wH\). \(US\) is a
result of Theorem 4 from \cite{MR0208557}:
the one-point compactification of any \(KC\) space
is \(US\).

The space is not \(wH\) because it contains
a compact \(T_2\) space which is not closed
(since inclusion is a continuous function). Namely,
\(Y\setminus\{z\}=\omega^2\cup\{\infty\}\)
is a copy of the one-point
compactification of a discrete countable space; equivalently,
a copy of a non-trivial converging sequence in \(\mathbb R\).
But \(z\) is a limit point of this
set, so it is not closed.
\end{example}

\section{\(k\)-Hausdorff, revisited}\label{k2Hsec}

We have established the existence of multiple weakly Hausdorff
spaces that are not \(k\)-Hausdorff (S29, S31, S97). So
the following quote from \cite{rezkcompactly} may be surprising:
``Every weak Hausdorff space is \(k\)-Hausdorff.''

Regrettably, there are multiple inequivalent notions of ``\(k\)-Hausdorff''
in the literature. We will first establish the pattern they all appear 
to follow.

\begin{definition}\label{kiH}
A space \(X\) is said to be \term{\(k_i\)-Hausdorff} or \(k_iH\) provided that
the diagonal \(\Delta_X=\{(x,x):x\in X\}\) is \(k_i\)-closed in the product topology
on \(X^2\).
\end{definition}

In turn, we are now obligated to define \(k_i\)-closed.

\begin{definition}
A subset \(C\) of a space is \term{\(k_1\)-closed} provided for every
compact subset \(K\) of the space, the intersection \(C\cap K\)
is closed in the subspace topology for \(K\).
\end{definition}

\begin{theorem}[2.1 of \cite{MR0374322}]
The criteria for \(k_1\)-Hausdorff given in Definitions \ref{k1H} and \ref{kiH} are
equivalent.
\end{theorem}

So the notion of \(k\)-Hausdorff from \cite{rezkcompactly} is what we
would call \(k_2H\) according to the following definition.

\begin{definition}
A subset \(C\) of a space \(X\) is \term{\(k_2\)-closed} provided for every 
compact Hausdorff space \(K\) and continuous map \(f:K\to X\), \(f^\leftarrow[C]\)
is closed in \(K\).\footnote{
A set is \(k_3\)-closed if its intersection with each compact Hausdorff
subspace is closed in the subspace, but even indiscrete spaces have a
\(k_3\)-closed diagonal, so we do not investigate this topic further here.
}
\end{definition}

We will find the following characterization for \(k_2H\) convenient,
as like \(k_1H\)'s Definition \ref{k1H}, it does not require
considering the diagonal of the square.

\begin{theorem}[4.2.4 of \cite{rezkcompactly}]
A space \(X\) is \(k_2H\) if and only if for every 
compact Hausdorff space \(K\), continuous map \(f:K\to X\),
and points \(k_0,k_1\in K\) such that \(f(k_0)\not=f(k_1)\),
there exist open neighborhoods \(U_0,U_1\) of \(k_0,k_1\)
such that \(f[U_0]\cap f[U_1]=\emptyset\).
\end{theorem}

We now may establish this notion of \(k\)-Hausdorff as strictly
weaker than the property explored earlier.

\begin{theorem}
    \[wH\Rightarrow k_2H\Rightarrow US\]
\end{theorem}
\begin{proof}
The first implication is Proposition 11.2 of \cite{rezkcompactly}: given
\(f:K\to X\) and \(k_0,k_1\in K\) with \(f(k_0)\not=f(k_1)\), by Lemma \ref{wHimage}
\(f[K]\) is \(T_2\). So there exist disjoint open neighborhoods \(V_0,V_1\)
of \(f(k_0),f(k_1)\), and \(U_0=f^\leftarrow[V_0],U_1=f^\leftarrow[V_1]\) are open
neighborhoods of \(k_0,k_1\) with \(f[U_0]\cap f[U_1]=V_0\cap V_1=\emptyset\).

Now let \(X\) be \(k_2H\) and let $l_0,l_1$ be limits of $x_n\in X$. Let $K=(\omega+1)\times 2$, and let $f:K\to X$ be defined by $f(n,i)=x_n$, $f(\omega,0)=l_0$, and $f(\omega,1)=l_1$. To show this is continuous, we need only observe that inverse images of open subsets of $X$ that contain $(\omega,i)$ contain a cofinite subset of $\omega\times\{i\}$. This follows as for an inverse open image to contain $(\omega,i)$, the open set must contain the limit $l_i$ of $x_n$, and thus contain a final sequence of $x_n$, and thus the inverse open image contains a cofinite subset of $\omega\times 2$.

Finally, if $l_0=f(\omega,0)\not=f(\omega,1)=l_1$, then there would exist open neighborhoods $U_0,U_1$ of $(\omega,0),(\omega,1)$ with $f[U_0]\cap f[U_1]=\emptyset$. But this is impossible as there exists $n<\omega$ with $(n,0)\in U_0$ and $(n,1)\in U_1$, and $f(n,0)=x_n=f(n,1)$. Thus $l_0=l_1$, showing limits are unique and \(X\) is \(US\).
\end{proof}

These arrows do not reverse. For the second:

\begin{example}[S37]\label{S37}
Consider the space \(X=(\omega_1+1)\cup\{\omega_1'\}\), where $\omega_1+1$ has its order topology and 
$\omega_1'$ is a duplicate of \(\omega_1\).

To see that this space fails $k_2H$, consider the set $(\omega_1+1)\times 2$ and the map $(\alpha,i)\mapsto \alpha$ for $\alpha<\omega_1$, $(\omega_1,0)\mapsto\omega_1$, and $(\omega_1,1)\mapsto\omega_1'$. This map is continous, $f(\omega_1,0)\not=f(\omega_1,1)$, but there are no open neighborhoods $U_0,U_1$ of $(\omega_1,0),(\omega_1,1)$ with $f[U_0]\cap f[U_1]=\emptyset$.

The fact that this space is \(US\) follows from the fact that \(\omega_1+1\) is \(US\),
and no non-trivial sequences in \(\omega_1+1\) converge to \(\omega_1\).
\end{example}

To reject the converse of \(wH\Rightarrow k_2H\), we will use the following
lemma, which strengthens \cite[Theorem 4]{MR0208557} used earlier.

\begin{lemma}
The one-point compactification of a $KC$ space is $k_2H$.
\end{lemma}
\begin{proof}
Let $X^+=X\cup\{\infty\}$ be the one-point compactification of $X$. Consider a compact Hausdorff $K$ with $f:K\to X^+$, and $k_0,k_1\in K$ with $f(k_0)\not= f(k_1)$. In the case that $f(k_t)\not=\infty$ for $t\in\{0,1\}$, consider the closed subset $f^\leftarrow[\{\infty\}]$ of $K$. Since $K$ is regular, we may choose open $U$ with $\{k_0,k_1\}\subseteq U\subseteq cl(U)\subseteq K\setminus f^\leftarrow[\{\infty\}]$. Then $f\upharpoonright cl(U):cl(U)\to X$ is a continuous map from compact Hausdroff $cl(U)$ to $KC$ and therefore $k_2H$ $X$, so there exist open (in $cl(U)$) neighborhoods $V_0,V_1$ of $k_0,k_1$ with $f[V_0]\cap f[V_1]=\emptyset$. It follows that $U_0=V_0\cap U,U_1=V_1\cap U$ are open neighborhoods of $k_0,k_1$ in $K$, and $f[U_0]\cap f[U_1]\subseteq f[V_0]\cap f[V_1]=\emptyset$.

Otherwise we have, say, $k_0$ with $f(k_0)=\infty$. Since $X^+$ is $T_1$, $\{\infty\}$ is closed and $f^\leftarrow[\{\infty\}]$ is closed. Thus $V_{1}=K\setminus f^\leftarrow[\{\infty\}]$ is an open neighborhood of $k_{1}$. Since $K$ is regular, choose $U_{1}$ open with $k_{1}\in U_{1}\subseteq cl(U_{1})\subseteq V_{1}= K\setminus f^\leftarrow[\{\infty\}]$. Then $cl(U_{1})$ is compact and thus $f[cl(U_{1})]$ is compact and misses $\infty$, and thus is a compact and closed subset of $X$. It follows that $W_0=X^+\setminus f[cl(U_{1})]$ is an open neighborhood of $\infty=f(k_0)$, so $U_0=f^\leftarrow[W_0]$ is an open neighborhood of $k_0$.

So we have open neighborhoods $U_0,U_{1}$ for $k_0,k_{1}$, and $f[U_0]\cap f[U_{1}]\subseteq f[cl(U_0)]\cap(X^+\setminus f[cl(U_0)])=\emptyset$.
\end{proof}

\begin{example}[S165]
The one-point compactification of the Arens-Fort space (S23, Example \ref{S23}) was shown
earlier to fail \(wH\), but satisfies \(k_2H\)
by the preceding lemma.
\end{example}

We now have the following corollary.

\begin{corollary}\label{impl-chain}
\[T_2\Rightarrow k_1H\Rightarrow KC\Rightarrow
wH\Rightarrow k_2H\Rightarrow US\Rightarrow T_1\]
with no implications reversing.
\end{corollary}

\section{Retracts are Closed}

In \cite{191016} the following property was introduced.

\begin{definition}
A \term{retract} \(R\) of a space \(X\) is a subspace for which
a continuous \(f:X\to R\) exists with \(f(r)=r\) for all
\(r\in R\).
\end{definition}

\begin{definition}[P101]
A space is \term{RC} provided each retract subspace
is closed.
\end{definition}

This provides another example of a property between
\(T_1\) and \(T_2\).

\begin{theorem}
\[T_2\Rightarrow RC\Rightarrow T_1\]
\end{theorem}
\begin{proof}
Given a \(T_2\) space \(X\), and consider a retract \(R\subseteq X\)
and witnessing map \(f:X\to R\). Given \(x\in X\setminus R\), note
\(f(x)\not=x\). So choose open \(U,V\) with \(x\in U,f(x)\in V\).

Consider now \(W=U\cap f^\leftarrow[V]\). This is an open neighborhood
of \(x\). Furthermore, for each \(r\in R\cap U\), \(f(r)=r\in U\),
and thus \(f(r)\not\in V\). Thus
\(R\cap W=R\cap U\cap f^\leftarrow[V]=\emptyset\), showing \(W\)
is an open neighborhood of \(x\) missing \(R\). Thus \(R\)
is closed.

Finally let \(X\) be \(RC\). Then for each \(x\in X\), note
\(f:X\to\{x\}\) is continuous, and thus \(\{x\}\) is closed,
showing \(X\) is \(T_1\).
\end{proof}

However, this property cannot be placed into the chain of
implications given in Corollary \ref{impl-chain}.

\begin{example}
In \cite{MR4614765}, the authors provide an example of a
space which is \(RC\) (in fact, ``strongly rigid'' as
all continuous self-maps are either constant or the identity,
see their Claim 4.9),
but is not \(KC\). An earlier version of this construction,
available within arXiv preprint \verb|2211.12579v3|,
was in fact not \(US\).
\end{example}

We do not delve into the details of the preceding construction
as they are fairly technical. The author suspected that
a more elementary example could be constructed, which
was eventually provided by Marshall Williams in a response
to the author's post to MathOverflow \cite{460346}:

\begin{example}[S192]
Let $X=[0,\infty)\cup \{\infty_1,\infty_2\}$
where $[0,\infty)$ is an open subspace with its usual
Euclidean topology, $\infty_1$ has neighborhoods of the form 
$\{\infty_1\}\cup (a,\infty)\backslash 2\mathbb Z$, 
and $\infty_2$ has neighborhoods of the form 
$\{\infty_2\}\cup (a,\infty)\backslash (2\mathbb Z +1)$.
Observe that this space is connected, and that
it fails to be \(US\) as the sequence \(a_n=n+\frac{1}{2}\) 
converges to both $\infty_1$ and $\infty_2$.

Note then that if $A$ is a retract of $X$, then
$\operatorname{cl}(A)\cap [0,\infty)=A\cap [0,\infty)$.
If $\infty_1$ is a limit point of $A$, we have that
$A\setminus 2\mathbb Z$ is unbounded.
Furthermore, since
$A$ is the continuous image of connected $X$, we have
$A\supseteq[a,\infty)$ for some $a$.
Thus $\infty_1\in\operatorname{cl}(A)$ implies $\infty_1\in A$.
A similar argument shows
$\infty_2\in\operatorname{cl}(A)$ implies $\infty_2\in A$,
and therefore $\operatorname{cl}(A)=A$.

% Suppose now that $\infty_1\in \overline{A}$. Then either $\infty_1\in A$, or $A$ contains an unbounded subset of $[0,\infty)$. 

% In the latter case, since $X$ is connected, the retract $A$ is connected, so $A\supseteq [a,\infty)$ for some $a$.  Therefore the sequence $2n+1$ eventually lies in $A$, so since $2n+1$ has a unique limit of $\infty_1$, we again have $\infty_1\in A$.

% The same argument applies to $\infty_2$ via the sequence $2n$, and so the retract $A$ is closed.  

% Thus $X$ is **RC**.

\end{example}

We can see that \(k_1H\) is not sufficient to imply
\(RC\).

\begin{example}[S17, Example \ref{S17}]
Earlier we noted that
the co-countable topology on an uncountable
set is \(k_1H\) but not \(T_2\). In fact,
the space fails to be \(RC\): partition this space
into pairs of points, and consider a map that
sends each pair to a single point within
their set. Then this image is uncountable and
therefore not closed. But this image is a retract as
the map is continuous: given a co-countable set,
the inverse image is still co-countable.
\end{example}

It's important that S17 is not compact, given the following
result.

\begin{proposition}
All compact \(KC\) spaces are \(RC\).
\end{proposition}
\begin{proof}
Every retract of a compact space is compact,
so every retract of a compact \(KC\) space is closed.
\end{proof}

But \(KC\) cannot be weakened to \(wH\) here.

\begin{proposition}
The square of a non-\(T_2\) space is not \(RC\).
\end{proposition}
\begin{proof}
The diagonal is a retract of the square by the continuous
map \((x,y)\mapsto(x,x)\), and a space is \(T_2\) if and
only if its diagonal is closed.
\end{proof}

Note that the previous two propositions provide an alternative
proof for Theorem \ref{non-kc-square}.

\begin{example}[S31, Example \ref{S31}]
The square of the one-point compactification of the rationals
was noted earlier to be \(wH\), but fails to be \(RC\)
by the preceding proposition.
\end{example}

We leave open the following question:

\begin{question}
Does there exist a compact \(RC\) space which is not
\(KC\)?
\end{question}

\section{Semi-Hausdorff}

The semi-Hausdorff property was introduced in \cite{banakh2017each}.

\begin{definition}
An open set is \term{regular open} if it equals the interior of its
closure.
\end{definition}

\begin{definition}[P10]
A space is \term{semiregular} if it has a basis of
regular open sets.
\end{definition}

\begin{definition}[P169]
A space is \term{semi-Hausdorff} or \(sH\) if for each
pair of distinct points \(x,y\), there exists a
regular open neighborhood of \(x\) missing \(y\).
\end{definition}

\begin{proposition}
All \(T_1\) semiregular spaces are semi-Hausdorff.
\end{proposition}
\begin{proof}
Given \(T_1\), there exists an open neighborhood \(U\)
of \(x\) missing \(y\). Then by semiregular, \(U\) contains
a regular open neighborhood \(V\) of \(x\) missing \(y\).
\end{proof}

\begin{theorem}
\[T_2\Rightarrow sH\Rightarrow T_1\]
\end{theorem}
\begin{proof}
Given points \(x,y\), let \(U,V\) be disjoint open neighborhoods.
Then \(\cl(U)\cap V=\emptyset\),
so \(\int(\cl(U))\) is a regular open neighborhood of \(x\) missing 
\(y\). The second implication is immediate from the definitions.
\end{proof}

\begin{example}
The \(RC\)-not-\(KC\) space constructed in \cite{MR4614765}
is \(sH\). Likewise, its arXiv
arXiv preprint \verb|2211.12579v3|
includes a construction which is \(RC\), \(sH\), but not \(US\).
\end{example}

In this case we may return to a more straightforward example
to show not all \(sH\) spaces are \(T_2\).

\begin{example}[S97, Example \ref{S97}]
Earlier this space was shown to be \(KC\), and therefore
\(T_1\). The space is also compact and thus \(RC\).

Singletons
are regular open neighborhoods for points of \(\omega^2\).
The open neighborhoods of \(\infty_y\) missing at least the first row of
\(\omega^2\) are regular open, and the open neighborhoods
of \(\infty_x\) missing at least the first column of \(\omega^2\)
are regular open.

Since the space is semiregular and
\(T_1\), the space is \(sH\).
(Note that the authors of \cite{steen1978counterexamples}
assume \(T_2\) in their
definition for semiregular, explaining why it is marked
as false in their reference chart.)
\end{example}

\begin{example}[e.g. S37, Example \ref{S37}]
Consider any \(T_2\) space \(X\) with a non-isolated
point \(x\in X\).
Let \(X'=X\cup\{x'\}\) have the topology generated by
the original open subsets of \(X\), and for each open
neighborhood \(U\) of \(x\), let \(\{x'\}\cup U\setminus\{x\}\)
be a basic open neighborhood of \(x'\).

This space remains \(T_1\), but it can be seen immediately that
\(x,x'\) cannot be separated by open sets. In fact, this space
cannot be even \(sH\) as every regular open neighborhood of
\(x\) contains \(x'\): given a neighborhood \(U\) of \(x\),
its closure contains the open neighborhood \(\{x'\}\cup U\) of
\(x'\). Likewise, the space is not \(RC\): the map
\(x'\mapsto x\) which is otherwise the identity shows
\(X\) is a retract of \(X'\) which is not closed.
\end{example}

Earlier it was noted that S37 is \(US\), but not \(k_2H\).
In fact, we've already seen an
example that shows that \(k_1H\) is insufficient to imply
\(sH\).

\begin{proposition}
The only nonempty hyperconnected \(sH\) space is the
singleton.
\end{proposition}
\begin{proof}
Let \(x,y\) be distinct points of a hyperconnected
space. Then every regular open neighborhood of \(x\)
is the entire space, and thus cannot be used to separate
\(x\) from \(y\).
\end{proof}

\begin{example}[S17, Example \ref{S17}]
Earlier we noted that
the co-countable topology on an uncountable
set is \(k_1H\) but not \(RC\).
The space also fails to be \(sH\) by
the preceding proposition.
\end{example}

We've now established that, like \(RC\),
\(sH\) cannot be placed into the chain of implications
given in Corollary \ref{impl-chain}: S17 is \(k_1H\) but
neither \(sH\) nor \(RC\), while the example from
arXiv preprint \verb|2211.12579v3| is \(sH\)
and \(RC\), but not even \(US\). (We will introduce a
more elementary example of an \(sH\) space that fails
to be \(US\) as Example \ref{S185} below.)

So how might \(RC\) and \(sH\) be related? We again
revisit a counter-example that will show us
that \(RC\) need not imply
\(sH\), even when spaces are compact.

\begin{definition}
A space is \term{nowhere locally compact} if no point of
the space has a compact neighborhood.
\end{definition}

\begin{theorem}
The one-point compactification of any
nowhere locally compact space \(X\) is not \(sH\).
\end{theorem}
\begin{proof}
Consider a regular open neighborhood \(U\) of the compactifying
point \(\infty\). Since \(X\) is nowhere locally compact, every
neighborhood of a point in \(x\) cannot miss \(U\). It follows
that \(\cl(U)=X\cup\{\infty\}\), and thus
\(U=\int(\cl(U))=\int(X\cup\{\infty\})=X\cup\{\infty\}\). Therefore \(\infty\)
cannot be separated from any other point by a regular open set.
\end{proof}

\begin{example}[S29, Example \ref{S29}]
The one-point compactification of the rationals
is \(KC\), and therefore \(RC\).
But the space fails to be \(sH\) by the previous
theorem: compact subsets
of the rationals are nowhere dense and therefore cannot contain
an open set,
showing the rationals are nowhere locally compact.
\end{example}

We may also obtain an example of a compact space which is
\(sH\) but neither \(RC\) nor \(US\).

\begin{example}[S185]\label{S185}
We introduce a coarsening of the topology of S97 on
\(\omega^2\cup\{\infty_x,\infty_y\}\). We continue
to assume \(\omega^2\) is discrete,
and neighborhoods of \(\infty_y\) contain all but
finitely-many rows, but neighborhoods of \(\infty_x\)
must contain all but finitely-many columns.

The argument that this space is \(sH\) is identical to
the argument for S97, and it's also immediate to verify
that the space is compact. However, unlike S97, this space is
neither \(RC\) nor \(US\). To see the latter, note
that the sequence \((n,n)\) converges to both \(\infty_x\)
and \(\infty_y\).

As for the former, we will show that the non-closed set
\(\{(n,n):n<\omega\}\cup\{\infty_y\}\) is a retract.
This set is the image of the map
\((a,b)\mapsto(\max\{a,b\},\max\{a,b\})\),
\(\infty_x\mapsto\infty_y\), and \(\infty_y\mapsto\infty_y\).
To see that this map is continuous, we need only consider
the preimage of basic neighborhoods of \(\infty_y\). Such
a neighborhood is of the form
\(\{(n,n):N\leq n<\omega\}\cup\{\infty_y\}\), and its preimage
is equal to
\(\{(n,m):N\leq n<\omega,m<\omega\}\cup\{\infty_x\}
\cup\{(m,n):N\leq n<\omega,m<\omega\}\cup\{\infty_y\}\)
and is thus open.
\end{example}

\section{Locally Hausdorff}

The final property we will investigate was studied in e.g.
\cite{niefield1983note}.

\begin{definition}[P84]
A space is \term{locally Hausdorff} or \(lH\) provided
every point has a \(T_2\) neighborhood.
\end{definition}

\begin{theorem}
\[T_2\Rightarrow lH\Rightarrow T_1\]
\end{theorem}
\begin{proof}
The first implication is immediate. For the second,
take distinct points \(x,y\) and let \(U\) be an open \(T_2\)
neighborhood of \(x\). If \(U\) misses \(y\), we're done;
otherwise use the \(T_2\) property of \(U\) to obtain a smaller
open neighborhood \(V\) of \(x\) missing \(y\).
\end{proof}

We may quickly verify these do not reverse by returning to
previously explored counterexamples.

\begin{example}[S17, Example \ref{S17}]
The co-countable topology on an uncountable set is
\(k_1H\), but neither \(RC\) nor \(sH\). It is also
not \(lH\) as every nonempty open set is homeomorphic to the whole
non-\(T_2\) space.
\end{example}

\begin{example}[S29, Example \ref{S29}]
The one-point compactification of the rationals is
\(KC\) and \(RC\), and thus \(T_1\), but not \(sH\). It is also not
\(lH\) as no neighborhood of \(\infty\) is \(T_2\).
\end{example}

\begin{example}[S185, Example \ref{S185}]
This space is \(sH\), but neither \(US\) nor \(RC\).
It is \(lH\), and in fact locally metrizable,
as neighborhoods of \(\infty_x\) and \(\infty_y\) that
miss each other are copies of the subspace
\(\{(0,0)\}\cup\left\{\left(\frac{\cos(1/m)}{n},\frac{\sin(1/m)}{n}\right)
:0<m,n<\omega\right\}\subseteq\mathbb R^2\).
\end{example}

\begin{example}[e.g. S37, Example \ref{S37}]
By doubling a non-isolated point of a \(T_2\) space we obtain a space
which is neither \(RC\) nor \(sH\), but remains \(lH\).
\end{example}

We have thus demonstrated that \(lH\) is incomparable with \(RC\),
and fails to slot into the chain of implications in
Corollary \ref{impl-chain}. We've also shown that \(lH\)
does not imply \(sH\). So we proceed to construct a counterexample
to the converse as well.

\begin{example}[e.g. S186]
Let \(X\) be an \(sH\) but non-\(T_2\) space
(e.g. S97, S185), and let \(\omega\) have the discrete topology.
Our space is \((X\times\omega)\cup\{\infty\}\),
where \(X\times\omega\) is open with the product topology.
Neighborhoods of \(\infty\) contain \(X\times(\omega\setminus N)\)
for some \(N<\omega\).

Immediately we see that the space is not \(lH\) as every
neighborhood of \(\infty\) contains a copy of the non-\(T_2\) space
\(Y\).

To see that the space is \(sH\), first consider a point
\((x,n)\in X\times\omega\). Given \(y\in X\setminus\{x\}\),
choose regular open \(U\subseteq X\) that contains \(x\) but
misses \(y\). Then \(U\times\{n\}\) is a regular open set that
contains \((x,n)\) but misses \((y,n)\). Given \(y\in X\) and
\(m\in\omega\setminus\{n\}\), \(X\times\{n\}\) is a regular
open set that contains misses \((y,m)\). Likewise, it misses
\(\infty\).

Finally, consider \(\infty\). Given any \((x,n)\in X\times\omega\),
\((X\times(\omega\setminus(n+1)))\cup\{\infty\}\) is a regular
open set containing \(\infty\) that misses \((x,n)\).
\end{example}

\section{Conclusion}

The results of this paper are summarized in Figure \ref{fig-summary};
note that we have provided counter-examples for every property pair
not shown to have a chain of implications in this figure.

To easily
query these counter-examples, the \(\pi\)-Base includes an Explore
page 
(as shown in Figure \ref{fig-pibase})
that allows for revealing which spaces in the database
satsify or fail the user's desired properties. Future work of the
\(\pi\)-Base will include the generation of diagrams such as
Figure \ref{fig-summary} automatically from its database, and
the generation of open (to the \(\pi\)-Base) questions which
we anticipate will be particularly helpful for developing
projects appropriate for undergraduate and graduate students.

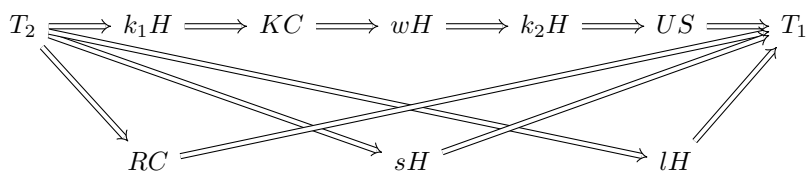
\begin{figure}
\caption{
Summary of results.
}\label{fig-summary}
% https://tikzcd.yichuanshen.de/#N4Igdg9gJgpgziAXAbVABwnAlgFyxMJZABgBpiBdUkANwEMAbAVxiRABUB9AJhAF9S6TLnyEUARnJVajFmwDWncQAl+gkBmx4CRblOr1mrRCADSAYTVCtoogGZ9Mo2wDuqgdZE6UAFkeG5E0Vud3VNLzFkAFZ-WWMQAFUAZSsNYW1IgDZY5xMucVTwjKJJbmkA+IAlSw80m29kBzKDOLY4UM9ilBjmp0CQBndpGCgAc3giUAAzACcIAFskMhAcCCQAdmoGLDB4qAgmACMGVmoACxg6KCQwJgYGahw6LAY2SF3U2YWlx7XESQGOz2B2OpxAFyuNzuDxWz1eJnerFqX0W-1+SD0gI+Jn2RxOIHOl2uiFu90ecLeBCR6hRGPRiAcWOBeLBEOJpJhTxelI+yLmqMZqyQfiZbFxoIJ4KJULJsO5CKpn35wvpMVFOJB+MJkJJ0PJ8vAir530QaqFiGy6pA4q1Up1HP18MNvJpysQmxWf0t22x1s1rOlutlXKdiKVJoAHPTvUCxf7JWyZZyKQqXdM3ctzVGrTaA-a9XLQ0bXSbM38AJxbWMalkJwMOws86npk2Vz1IGO+3N1-PBlPOpEUPhAA
\begin{center}
\begin{tikzcd}
T_2 \arrow[rdd, Rightarrow] \arrow[r, Rightarrow] \arrow[rrrdd, Rightarrow] \arrow[rrrrrdd, Rightarrow] & k_1H \arrow[r, Rightarrow]     & KC \arrow[r, Rightarrow] & wH \arrow[r, Rightarrow]     & k_2H \arrow[r, Rightarrow] & US \arrow[r, Rightarrow]   & T_1 \\
                                                                                                        &                                &                          &                              &                            &                            &     \\
                                                                                                        & RC \arrow[rrrrruu, Rightarrow] &                          & sH \arrow[rrruu, Rightarrow] &                            & lH \arrow[ruu, Rightarrow] &    
\end{tikzcd}
\end{center}
\end{figure}

\begin{figure}
\caption{
The Explore feature of \(\pi\)-Base}\label{fig-pibase}
\includegraphics[width=5in]{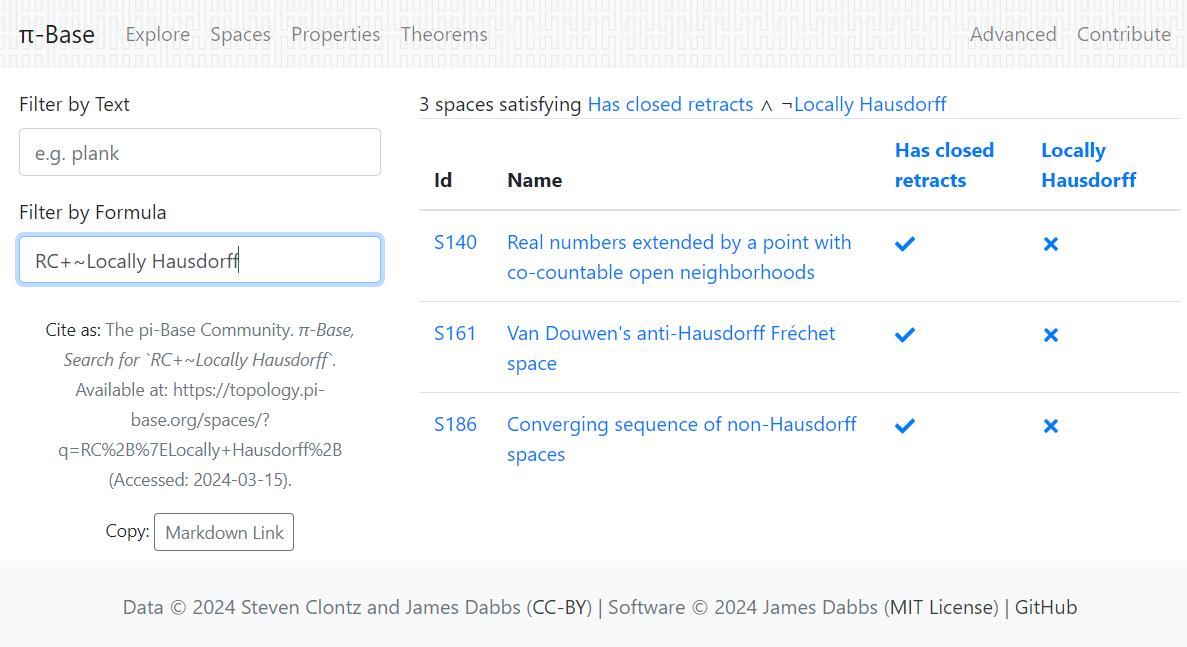}
\end{figure}

% \section{Acknowledgements}

% The author wishes to thank the \(\pi\)-Base, Math StackExchange,
% and MathOverflow communities, and Patrick Rabau in particular,
% for their comments on several of the results of this paper
% as they were aired out in various online discussions.